\documentclass[11pt,reqno]{amsart}
\usepackage[letterpaper,margin=1in,footskip=0.25in]{geometry}
\usepackage{microtype}
\usepackage{mathtools}
\usepackage{enumitem}
\usepackage[noadjust]{cite}
\renewcommand{\citepunct}{;\penalty\citemidpenalty\ }
\usepackage{longtable,booktabs,multirow}
\usepackage[referable]{threeparttablex}
\usepackage{caption}

\usepackage{xcolor}
\PassOptionsToPackage{pdfusetitle,colorlinks,pagebackref}{hyperref}
\usepackage{bookmark}
\hypersetup{citecolor=[HTML]{2E7E2A},linkcolor=[HTML]{800006},urlcolor=[HTML]{8A0087}}
\usepackage[hyphenbreaks]{breakurl}

\numberwithin{equation}{section}

\newtheorem{lemma}[equation]{Lemma}
\newtheorem{proposition}[equation]{Proposition}

\newtheorem{alphthm}{Theorem}

\newtheoremstyle{cited}{.5\baselineskip\@plus.2\baselineskip\@minus.2\baselineskip}{.5\baselineskip\@plus.2\baselineskip\@minus.2\baselineskip}{\itshape}{}{\bfseries}{\bfseries .}{5pt plus 1pt minus 1pt}{\thmname{#1}\thmnumber{ #2}\thmnote{ \normalfont#3}}
\theoremstyle{cited}
\newtheorem{citedprop}[equation]{Proposition}

\theoremstyle{definition}
\newtheorem{axiom}[equation]{Axiom}
\newtheorem{notation}[equation]{Notation}

\newtheoremstyle{citeddef}{.5\baselineskip\@plus.2\baselineskip\@minus.2\baselineskip}{.5\baselineskip\@plus.2\baselineskip\@minus.2\baselineskip}{}{}{\bfseries}{\bfseries .}{5pt plus 1pt minus 1pt}{\thmname{#1}\thmnumber{ #2}\thmnote{ \normalfont#3}}
\theoremstyle{citeddef}
\newtheorem{citeddef}[equation]{Definition}
\newtheorem{citedaxiom}[equation]{Axiom}
\newtheorem{citedaxioms}[equation]{Axioms}

\theoremstyle{remark}
\newtheorem{remark}[equation]{Remark}
\newtheorem{examples}[equation]{Examples}

\newtheoremstyle{step}{.25\baselineskip\@plus.1\baselineskip\@minus.1\baselineskip}{.25\baselineskip\@plus.1\baselineskip\@minus.1\baselineskip}{\itshape}{}{\bfseries}{\bfseries .}{5pt plus 1pt minus 1pt}{\thmname{#1}\thmnumber{ #2}\thmnote{ \normalfont(#3)}}
\theoremstyle{step}

\DeclareMathOperator{\Ass}{Ass}
\DeclareMathOperator{\Ext}{Ext}
\DeclareMathOperator{\Frac}{Frac}
\DeclareMathOperator{\Hom}{Hom}
\DeclareMathOperator{\Sym}{Sym}
\DeclareMathOperator{\im}{im}

\newcommand{\QQ}{\mathbf{Q}}
\newcommand{\ZZ}{\mathbf{Z}}
\newcommand{\cl}{\mathsf{cl}}
\newcommand{\epf}{\mathsf{epf}}
\newcommand{\ronef}{\mathsf{r1f}}
\newcommand{\wepf}{\mathsf{wepf}}
\newcommand{\fm}{\mathfrak{m}}
\newcommand{\fp}{\mathfrak{p}}

\providecommand\given{}
\newcommand\SetSymbol[1][]{\nonscript\:#1\vert\allowbreak\nonscript\:\mathopen{}}
\DeclarePairedDelimiterX\Set[1]\{\}{\renewcommand\given{\SetSymbol[\delimsize]}#1}

\begin{document}
\title[Closure-theoretic proofs of uniform bounds on symbolic powers in
regular rings]{Closure-theoretic proofs of\\uniform bounds on symbolic powers
in regular rings}
\author{Takumi Murayama}
\address{Department of Mathematics\\Princeton University\\Princeton, NJ
08544-1000\\USA}
\email{\href{mailto:takumim@math.princeton.edu}{takumim@math.princeton.edu}}
\urladdr{\url{https://web.math.princeton.edu/~takumim/}}

\thanks{This material is based upon work supported by the National Science
Foundation under Grant No.\ DMS-1902616}
\subjclass[2020]{Primary 13A15; Secondary 13H05, 14G45, 13A35, 13C14}

\keywords{uniform bounds, symbolic powers, regular rings,
closure operations, big Cohen--Macaulay algebras}

\makeatletter
  \hypersetup{
    pdfsubject=\@subjclass,pdfkeywords=\@keywords
  }
\makeatother

\begin{abstract}
  We give short, closure-theoretic proofs for uniform bounds on the growth of
  symbolic powers of ideals in regular rings.
  The author recently proved these bounds in mixed characteristic
  using various versions of perfectoid/big Cohen--Macaulay test ideals, with
  special cases obtained earlier by Ma and Schwede.
  In mixed characteristic, we instead use
  Heitmann's full extended plus ($\mathsf{epf}$) closure, Jiang's weak
  $\mathsf{epf}$ ($\mathsf{wepf}$) closure, and R.G.'s results on closure
  operations that induce big Cohen--Macaulay algebras.
  Our strategy also applies to any Dietz closure satisfying R.G.'s algebra
  axiom and a Brian\c{c}on--Skoda-type theorem, and hence
  yields new proofs of these results on uniform bounds on the
  growth of symbolic powers of ideals in regular rings of all
  characteristics.
  In equal characteristic, these results on symbolic powers are due to
  Ein--Lazarsfeld--Smith, Hochster--Huneke, Takagi--Yoshida, and Johnson.
\end{abstract}

\maketitle

\section{Introduction}\label{sect:intro}
Let $R$ be a Noetherian ring, and let $I \subseteq R$ be an ideal.
Following \cite[p.\ 349]{HH02},
for all integers $n \ge 1$, the \textsl{$n$-th symbolic power} of $I$ is the
ideal
\begin{equation}\label{eq:hhsymbdef}
  I^{(n)} \coloneqq I^nR_W \cap R,
\end{equation}
where $R_W$ is the localization of $R$ with respect to the multiplicative set $W
= R - \bigcup_{\fp \in \Ass_R(R/I)} \fp$.\medskip
\par In this paper, we give a short, closure-theoretic proof of the following
comparison between ordinary
and symbolic powers of ideals in regular rings.
Our proof strategy works in all characteristics.
\begin{alphthm}\label{thm:mainelshhms}
  Let $R$ be a regular ring, and let $I \subseteq R$ be an ideal.
  Let $h$ be the largest analytic spread of the ideals $IR_\fp \subseteq R_\fp$,
  where $\fp$ ranges over all associated primes $\fp$ of $R/I$.
  Then, for all $n \ge 1$, we have
  \begin{equation}\label{eq:elsbound}
    I^{(hn)} \subseteq I^n.
  \end{equation}
\end{alphthm}
\par Work of Schenzel
\citeleft\citen{Sch85}\citemid Theorem 1\citepunct \citen{Sch86}\citemid Theorem
3.2\citeright\ (see also \cite[\S2.1]{Sch98}) and Swanson \cite[Main Theorem
3.3]{Swa00} yields a complete characterization for when an ideal $I$ in an
arbitrary (not necessarily regular) ring satisfies \eqref{eq:elsbound}
for some integer $h$, but the integer $h$ appearing in their
results depends on the ideal $I$.
Theorem \ref{thm:mainelshhms} implies in particular that
when $R$ is a finite-dimensional regular ring, one can take $k = \dim(R)$ by
\cite[Proposition 5.1.6]{SH06}, and hence there is a uniform $k$ that does not
depend on $I$.
Schenzel's results answered a question of Hartshorne \cite[p.\ 160]{Har70},
special cases of which were studied previously by Zariski \cite[Lemma 3 on p.\
33]{Zar51}.
\par Theorem \ref{thm:mainelshhms} is due to Ein, Lazarsfeld, and Smith
\cite[Theorem 2.2 and Variant on p.\ 251]{ELS01} and
Hochster and Huneke \cite[Theorems 2.6 and 4.4$(a)$]{HH02} in equal
characteristic.
In this latter paper,
Hochster and Huneke asked whether Theorem
\ref{thm:mainelshhms} holds in mixed characteristic \cite[p.\ 368]{HH02}.
Ma and Schwede answered the special case of
Hochster and Huneke's question in the affirmative when 
$R$ is excellent (or, more generally, when $R$ has geometrically
reduced formal fibers) and $I$ is radical \cite[Theorem 7.4]{MS18}.
We recently answered Hochster and Huneke's question in complete generality
\cite[Theorem A]{Mur} using various versions of
perfectoid/big Cohen--Macaulay test
ideals.\medskip
\par The proof of Theorem \ref{thm:mainelshhms} in equal characteristic
due to Hochster and Huneke \cite{HH02} uses the theory of tight closure in
equal characteristic $p > 0$ \cite{HH90} or equal characteristic zero
\cite{HHchar0}.
This proof is remarkably short, especially in equal characteristic $p > 0$
when $I$ is radical \cite[pp.\ 350--351]{HH02}.
\par The goal of this paper is to give a short, closure-theoretic proof of
Theorem \ref{thm:mainelshhms} in mixed characteristic using an appropriate
replacement for tight closure.
In fact, our proof applies in any context where a sufficiently well-behaved
closure operation exists, and hence gives a new proof of Theorem
\ref{thm:mainelshhms} in all characteristics.
As we did in \cite{Mur}, we in fact show a stronger version of Theorem 
\ref{thm:mainelshhms}
involving various products of symbolic
powers on the right-hand side of \eqref{eq:elsbound},
and we also show a related uniform bound for regular local rings
related to a conjecture of Eisenbud and Mazur \cite[p.\ 190]{EM97}.
These results were shown by Johnson \cite[Theorems
4.4 and 4.3(2)]{Joh14} in equal characteristic using tight closure.
\begin{alphthm}\label{thm:elshhmsmb}
  Let $R$ be a regular ring, and let $I \subseteq R$ be an ideal.
  Let $h$ be the largest analytic spread of the ideals $IR_\fp \subseteq R_\fp$,
  where $\fp$ ranges over all associated primes $\fp$ of $R/I$.
  Then, for all $n \ge 1$ and for all non-negative integers $s_1,s_2,\ldots,s_n$
  with $s = \sum_{i=1}^n s_i$, we have
  \[
    I^{(s+nh)} \subseteq \prod_{i=1}^n I^{(s_i+1)}.
  \]
\end{alphthm}
\begin{alphthm}\label{thm:elshhmsmc}
  Let $(R,\fm)$ be a regular local ring, and let $I \subseteq R$ be an ideal.
  Let $h$ be the largest analytic spread of the ideals $IR_\fp \subseteq R_\fp$,
  where $\fp$ ranges over all associated primes $\fp$ of $R/I$.
  Then, for all $n \ge 1$ and for all non-negative integers $s_1,s_2,\ldots,s_n$
  with $s = \sum_{i=1}^n s_i$, we have
  \[
    I^{(s+nh+1)} \subseteq \fm\cdot\prod_{i=1}^n I^{(s_i+1)}.
  \]
\end{alphthm}
In equal characteristic, special cases of Theorem \ref{thm:elshhmsmb} (resp.\
Theorem \ref{thm:elshhmsmc}) are due to Hochster and Huneke \cite[Theorems 2.6
and 4.4$(a)$]{HH02} (resp.\ Hochster and Huneke \cite[Theorems 3.5 and
4.2(1)]{HH07} and Takagi and Yoshida \cite[Theorems 3.1 and 4.1]{TY08}).
The full statements are due to Johnson \cite[Theorems
4.4 and 4.3(2)]{Joh14} in equal characteristic and to us \cite[Theorems B and
C]{Mur} in mixed characteristic.
See \cite[p.\ 2]{Mur} for more discussion.\medskip
\par Our closure-theoretic strategy in this paper gives
shorter proofs of Theorems
\ref{thm:mainelshhms}, \ref{thm:elshhmsmb}, and \ref{thm:elshhmsmc}
in mixed characteristic compared to \cite{Mur}, where we used various versions
of perfectoid/big Cohen--Macaulay test ideals from
\cite{MS18,MS,PRG,MSTWW,Rob,ST,HLS}.
Our proofs in \cite{Mur} utilized
Bhatt's recent results on the Cohen--Macaulayness of absolute
integral closures in mixed characteristic \cite{Bha} and subsequent developments
from \cite{TY,BMPSTWW,HLS}.\medskip
\par Instead of using test ideals, our proof strategy for Theorems
\ref{thm:mainelshhms}, \ref{thm:elshhmsmb}, and \ref{thm:elshhmsmc} in mixed
characteristic is inspired by an unpublished proof of Theorem
\ref{thm:mainelshhms} in equal characteristic $p >0$ due to Hochster and Huneke
\cite{HH02}, which uses plus closure \cite{HH92,Smi94}.
Hochster and Huneke sketch their strategy
in \cite[p.\ 353]{HH02}.
Their strategy uses the following two ingredients:
\begin{enumerate}[label=$(\textup{\Roman*})$]
  \item Plus closure localizes \citeleft\citen{HH92}\citemid Lemma
    6.5$(b)$\citepunct \citen{Smi94}\citemid p.\ 45\citeright.
  \item A Brian\c{c}on--Skoda-type theorem holds for plus closure
    \cite[Theorem 7.1]{HH95}.
\end{enumerate}
\par In mixed characteristic, Heitmann introduced full extended plus ($\epf$)
closure \cite{Hei01} as a possible replacement for tight closure and plus
closure in mixed characteristic.
Following \cite{Hei01,HM21},
for a domain $R$ such that the image of $p$ lies in the Jacobson
radical of $R$, the \textsl{full extended plus} ($\epf$) \textsl{closure} of an
ideal $I \subseteq R$ is
\[
  I^\epf \coloneqq \Set*{x \in R \given \begin{tabular}{@{}c@{}}
    there exists $c \in R - \{0\}$ such that\\
    $c^\varepsilon x \in (I,p^N)R^+$\\
    for every $\varepsilon \in \QQ_{>0}$ and every $N \in \ZZ_{>0}$
  \end{tabular}}.
\]
Here, $R^+$ denotes the \textsl{absolute integral closure} of $R$, i.e., the
integral closure of a domain $R$ in an algebraic closure of
its fraction field \cite[p.\ 283]{Art71}.
\par Heitmann proved a Brian\c{c}on--Skoda-type theorem for
$\epf$ closure \cite[Theorem 4.2]{Hei01}, and
used $\epf$ closure to prove Hochster's direct summand
conjecture in mixed characteristic for rings of dimension three
\cite{Hei02}.
More recently, using
techniques due to Andr\'e \cite{And18a,And18b} and Bhatt \cite{Bha18}
stemming from Scholze's theory of perfectoid spaces \cite{Sch12},
Heitmann and Ma showed that Heitmann's Brian\c{c}on--Skoda-type theorem
for $\epf$ closure implies the Brian\c{c}on--Skoda theorem in mixed
characteristic \cite[Theorem 3.20]{HM21} and gave a closure-theoretic proof of
Hochster's direct summand conjecture in mixed characteristic
\cite[p.\ 135]{HM21}.
The Brian\c{c}on--Skoda theorem was originally proved by Brian\c{c}on and Skoda
\cite{SB74} for $\mathbf{C}\{z_1,z_2,\ldots,z_n\}$ and by Lipman and Sathaye
\cite{LS81} in general.
The direct summand conjecture was originally proved by Hochster \cite{Hoc73}
in equal characteristic and by Heitmann \cite{Hei02}
(in dimension three) and Andr\'e \cite{And18a,And18b} (in general) in mixed
characteristic.
\par Our strategy in this paper utilizes $\epf$ closure in
conjunction with Jiang's weak $\epf$
($\wepf$) closure \cite{Jia21} and R.G.'s results on closure operations that
induce big Cohen--Macaulay algebras \cite{RG18}.
We use $\wepf$ closure and R.G.'s results to work around the fact that
$\epf$ closure does not localize
\cite[p.\ 817]{Hei01}, and hence Hochster and Huneke's unpublished proof of
Theorem \ref{thm:mainelshhms} using plus closure in equal characteristic $p >0$
\cite[p.\ 353]{HH02} cannot readily be adapted to the mixed characteristic case.
Following \cite{Jia21},
for a domain $R$ as in the previous paragraph, the \textsl{weak
$\epf$} ($\wepf$) \textsl{closure} of an ideal $I \subseteq R$ is
\[
  I^\wepf \coloneqq \bigcap_{N=1}^\infty (I,p^N)^\epf.
\]
\par The advantage of using $\wepf$ closure is that it satisfies key
additional properties not known for $\epf$ closure,
namely Dietz's axioms \cite[Axioms
1.1]{Die10} and R.G.'s algebra axiom \cite[Axiom 3.1]{RG18}.
Dietz formulated the axioms in \cite{Die10} to characterize when a closure
operation can be used to construct big Cohen--Macaulay modules.
R.G. formulated the algebra axiom in \cite{RG18} to characterize when a closure
operation can be used to construct big Cohen--Macaulay algebras.
Using perfectoid techniques from \cite{And18a,And18b,Bha18,HM21},
Jiang \cite[Theorem
4.8]{Jia21} showed that $\wepf$ closure
satisfies Dietz's axioms and R.G.'s algebra axiom when
$R$ is a complete local domain of mixed characteristic with $F$-finite residue
field.
As a consequence,
we have the inclusions
\begin{equation}\label{eq:epfwepfbcm}
  I^\epf \subseteq I^\wepf \subseteq IB \cap R
\end{equation}
for some big Cohen--Macaulay algebra $B$ over $R$ by a result of R.G.
\cite[Proposition 4.1]{RG18}.
We pass to $IB \cap R$ instead of working with
$I^\epf$ or $I^\wepf$ because
there are no elements $c^\varepsilon$ or $p$-powers involved in the definition,
and hence we can detect when
nonzerodivisors in $R/I$ map to nonzerodivisors in $B/IB$.
Detecting when elements are nonzerodivisors is necessary to avoid localizations
when proving Theorems \ref{thm:mainelshhms}, \ref{thm:elshhmsmb}, and
\ref{thm:elshhmsmc}.\medskip
\par To illustrate our proof strategy for Theorems
\ref{thm:mainelshhms}, \ref{thm:elshhmsmb}, and \ref{thm:elshhmsmc}, we sketch
the proof of Theorem \ref{thm:mainelshhms} when $R$ is a complete regular 
local ring of
mixed characteristic with an $F$-finite residue field and $I$ is generated by
$h$ elements.
If $u \in I^{(hn)}$, there exists an element $c$ avoiding the associated primes
of $R/I$ such that $cu \in I^{hn}$.
We then consider a module-finite extension $R'$ of $R$ that is complete local
and contains an $n$-th root $u^{1/n}$ of $u$.
We have
\begin{align*}
  (cu^{1/n})^n &= c^nu \in (I^hR')^n,
\intertext{and hence $cu^{1/n} \in \overline{I^hR'}$.
By Heitmann's Brian\c{c}on--Skoda-type theorem for $\epf$ closure
\cite[Theorem 4.2]{Hei01} and using \eqref{eq:epfwepfbcm}, we know that}
  cu^{1/n} \in (IR')^\epf &\subseteq (IR')^\wepf \subseteq IB \cap R',
\end{align*}
where $B$ is a big Cohen--Macaulay algebra over $R'$.
Then, $B$ is also a big Cohen--Macaulay algebra over $R$ and
$R \to B$ is faithfully flat.
We therefore see that $c$
is a nonzerodivisor on $B/IB$, and hence $u^{1/n} \in IB \cap R'$.
Taking $n$-th powers, we have $u \in I^nB \cap R = I^n$.\medskip
\par Finally, our proof strategy also applies to any
closure operation satisfying Dietz's axioms, R.G.'s algebra axiom, and a
Brian\c{c}on--Skoda-type theorem.
We therefore obtain new proofs of Theorems \ref{thm:mainelshhms},
\ref{thm:elshhmsmb}, and \ref{thm:elshhmsmc} in all characteristics, since
tight closure \cite{HH90} and plus closure \cite{HH92,Smi94} satisfy these
properties in equal characteristic $p > 0$, and
$\mathfrak{B}$-closure \cite{AS07} satisfies these properties in equal
characteristic zero.
In mixed characteristic, one can alternatively use Heitmann's full rank $1$
($\ronef$) closure \cite{Hei01} instead of $\wepf$ closure
(still using results from \cite{RG18,Jia21}).
See Table \ref{tab:dietzclosures} for references to proofs of these axioms for
these closure operations.
We can also adapt our proofs to use small or big equational tight clossure in
equal characteristic zero \cite{HHchar0}.
See Remark \ref{rem:weakalgaxiom}.
\subsection*{Outline}
We review some preliminaries in
\S\ref{sect:prelims}.
We show that the definition of symbolic
powers in \eqref{eq:hhsymbdef} matches the definition used in \cite{Mur}
(Lemma \ref{lem:symbolicdefs}) and
define full extended plus ($\epf$) closure \cite{Hei01} and weak $\epf$
($\wepf$) closure \cite{Jia21}.
We state Dietz's axioms for closure operations \cite{Die10} and R.G.'s
algebra axiom \cite{RG18}, and collect references for proofs of these axioms and
for Brian\c{c}on--Skoda-type theorems for various closure operations in Table
\ref{tab:dietzclosures}.
We also state a result of R.G. from \cite{RG18} stating that Dietz
closures satisfying R.G.'s algebra axiom are contained in big Cohen--Macaulay
algebra closures (Proposition \ref{prop:rg41}).
Finally, we prove Theorems \ref{thm:mainelshhms}, \ref{thm:elshhmsmb}, and
\ref{thm:elshhmsmc} in \S\ref{sect:proofs}.
\subsection*{Conventions}
All rings are commutative with identity, and all ring maps are unital.
\par If $R$ is a domain, the
\textsl{absolute integral closure} $R^+$ of $R$ is the integral closure of
$R$ in an algebraic closure of the fraction field of $R$ \cite[p.\ 283]{Art71}.
\par If $(R,\fm)$ is a Noetherian local ring, then an $R$-algebra $B$
is a \textsl{(balanced) big Cohen--Macaulay algebra} over $R$ if every system of
parameters for $R$ is a regular sequence on $B$
\citeleft\citen{Hoc75queens}\citemid pp.\ 110--111\citepunct
\citen{Sha81}\citemid Definition 1.4\citeright.

\subsection*{Acknowledgments}
We are grateful to
Hailong Dao,
Elo\'{i}sa Grifo,
Melvin Hochster, 
Zhan Jiang,
Rebecca R.G.,
Farrah Yhee, and
Wenliang Zhang
for helpful conversations.

\section{Preliminaries}\label{sect:prelims}
\subsection{Symbolic powers}
For completeness,
we show that the definition of symbolic powers in \eqref{eq:hhsymbdef}
matches the definition used in \cite{Mur}.
This latter definition is the one used in the survey \cite[p.\
388]{DDSGHNB18}.
A similar argument will appear in the proof of Proposition
\ref{prop:elshhmsmd}$(\ref{prop:elshhmsmdbcm})$.
\begin{lemma}\label{lem:symbolicdefs}
  Let $R$ be a Noetherian ring, and let $I \subseteq R$ be an ideal.
  Denote by $R_W$ the localization of $R$ with respect to the multiplicative set
  $W = \bigcup_{\fp \in \Ass_R(R/I)} \fp$.
  Then, we have
  \begin{equation}\label{eq:diffdefs}
    I^nR_W \cap R = \bigcap_{\fp \in \Ass_R(R/I)} I^nR_\fp \cap R.
  \end{equation}
\end{lemma}
\begin{proof}
  The natural maps $R \to R_\fp$ factor through $R_W$ by the universal property
  of localization.
  We therefore have inclusions $I^nR_W \cap R \subseteq I^nR_\fp \cap R$ for
  every $\fp \in \Ass_R(R/I)$, which yields the inclusion ``$\subseteq$'' in
  \eqref{eq:diffdefs}.
  Thus, it suffices to show the inclusion ``$\supseteq$'' in
  \eqref{eq:diffdefs}.
  \par Let $u \in \bigcap_{\fp \in \Ass_R(R/I)} I^nR_\fp \cap R$.
  Denote by $\{\fp_\ell\}$ the subset of $\Ass_R(R/I)$ consisting of the
  associated primes that are maximal in $\Ass_R(R/I)$ with respect to inclusion.
  For each $\ell$, there exists an element $c_\ell \notin
  \fp_\ell$ such that $c_\ell u \in I^n$.
  For each $\ell$, we can also choose $d_\ell \in \fp_\ell - \bigcup_{j \ne
  \ell} \fp_j$ by prime avoidance.
  Now consider the element
  \begin{align*}
    c &= \sum_\ell \biggl( c_\ell \prod_{j \ne \ell} d_j \biggr).
  \intertext{Note that $c \notin \bigcup_\ell \fp_\ell$, for
  otherwise if $c \in \fp_{\ell_0}$ for some $\ell_0$, we have}
    c_{\ell_0} \prod_{j \ne \ell_0} d_j &=
    c - \sum_{\ell \ne \ell_0} \biggl( c_\ell \prod_{j \ne \ell} d_j \biggr) \in
    \fp_{\ell_0},
    \intertext{which contradicts the fact that $c_{\ell_0} \notin \fp_{\ell_0}$
    and $d_j \notin \fp_{\ell_0}$ for every $j \ne \ell_0$.
    Since $c_\ell u \in I^n$ by assumption, we see that}
    cu &= \sum_\ell \biggl( c_\ell u \prod_{j \ne \ell} d_j \biggr) \in I^n.
  \end{align*}
  Finally, since $c \in W = R - \bigcup_\ell \fp_\ell$,
  we have $u \in I^nR_W \cap R$.
\end{proof}

\subsection{Closure operations in mixed characteristic}
We recall the definition of Heitmann's full extended plus ($\epf$) closure.
\begin{citeddef}[{\citeleft\citen{Hei01}\citemid Definition on pp.\
  804--805\citepunct \citen{RG16}\citemid Definition 7.1\citepunct
  \citen{HM21}\citemid Definition 2.3\citepunct
  \citen{Jia21}\citemid Definition 2.3\citeright}]
  \label{def:epf}
  Let $p > 0$ be a prime number.
  Let $R$ be a domain such that the image of $p$ lies in the Jacobson
  radical of $R$.
  For an inclusion $Q \subseteq M$ of finitely generated $R$-modules,
  the \textsl{full extended plus} ($\epf$) \textsl{closure} of $Q$ in
  $M$ is
  \[
    Q^\epf_M \coloneqq \Set*{u \in M \given \begin{tabular}{@{}c@{}}
      there exists $c \in R - \{0\}$ such that\\
      $c^\varepsilon \otimes u \in \im(R^+ \otimes_R Q \to R^+ \otimes_R M) +
      p^N(R^+ \otimes_R M)$\\
      for every $\varepsilon \in \QQ_{>0}$ and every $N \in \ZZ_{>0}$
    \end{tabular}}.
  \]
\end{citeddef}
We also recall the definition of Jiang's weak $\epf$ ($\wepf$) closure.
\begin{citeddef}[{\cite[Definition 4.1]{Jia21}}]\label{def:wepf}
  Let $p > 0$ be a prime number.
  Let $R$ be a domain such that the image of $p$ lies in the Jacobson
  radical of $R$.
  For an inclusion $Q \subseteq M$ of $R$-modules,
  the \textsl{weak $\epf$} ($\wepf$) \textsl{closure} of $Q$ in $M$ is
  \[
    Q^\wepf_M \coloneqq \bigcap_{N = 1}^\infty (Q + p^NM)^\epf_M.
  \]
  By definition, we have the inclusion $Q^\epf_M \subseteq Q^\wepf_M$
  (see \cite[Remark 4.2]{Jia21}).
\end{citeddef}

\subsection{Axioms for closure operations}
In this subsection, we fix the following notation.
\begin{notation}\label{notation:closure}
  We denote by $R$ a ring, and by
  $Q$, $M$, and $W$ arbitrary finitely generated
  $R$-modules such that $Q \subseteq M$.
  We consider an operation $\cl$ sending submodules $Q \subseteq M$ to an
  $R$-module $Q_M^\cl$.
  We denote $I^\cl \coloneqq I^\cl_R$ for ideals $I \subseteq R$.
\end{notation}
We state Dietz's axioms for closure operations from \cite{Die10} with
some conventions from \citeleft\citen{Eps12}\citemid Definition 2.1.1\citepunct
\citen{Die18}\citemid Definition 1.1\citeright.
\begin{citedaxioms}[{\citeleft\citen{Die10}\citemid Axioms 1.1\citeright}]
  \label{axioms:dietzrg}
  Fix notation as in Notation \ref{notation:closure}.
  We say that the operation $\cl$ is a \textsl{closure operation}
  if it satisfies the following three axioms:
  \begin{enumerate}[label=$(\arabic*)$,ref=\arabic*]
    \item\label{axioms:dietz1}
      (Extension) $Q_M^\cl$ is a submodule of $M$ containing $Q$.
    \item (Idempotence) $(Q_M^\cl)^\cl = Q_M^\cl$.
    \item\label{axioms:dietz3}
      (Order-preservation) If $Q \subseteq M \subseteq W$, then $Q_W^\cl
      \subseteq M_W^\cl$.
  \end{enumerate}
  We also consider the following axioms:
  \begin{enumerate}[resume,label=$(\arabic*)$,ref=\arabic*]
    \item (Functoriality) Let $f\colon M \to W$ be a map of $R$-modules.
      Then, $f(Q_M^\cl) \subseteq f(Q)^\cl_W$.
    \item\label{axioms:dietz5}
      (Semi-residuality) If $Q^\cl_M = Q$, then $0^\cl_{M/Q} = 0$.
  \end{enumerate}
  Now suppose that $R$ is a Noetherian local domain $(R,\fm)$.
  We say that $\cl$ is a \textsl{Dietz closure} if it satisfies axioms
  $(\ref{axioms:dietz1})$--$(\ref{axioms:dietz5})$ above, and the following two
  additional axioms:
  \begin{enumerate}[resume,label=$(\arabic*)$,ref=\arabic*]
    \item The maximal ideal $\fm$ and the zero ideal $0$ are $\cl$-closed in
      $R$, i.e., $\fm^\cl = \fm$ and $0^\cl = 0$.
    \item (Generalized colon-capturing) Let $x_1,x_2,\ldots,x_{k+1}$ be a
      partial system of parameters for $R$ and let $J = (x_1,x_2,\ldots,x_k)$.
      Suppose there exists a surjective map $f\colon M \to R/J$ of $R$-modules,
      and let $v \in M$ be an arbitrary element
      such that $f(v) = x_{k+1} + J$.
      Then,
      \[
        (Rv)^\cl_M \cap \ker(f) \subseteq (Jv)^\cl_M.
      \]
  \end{enumerate}
\end{citedaxioms}
To state R.G.'s algebra axiom, we first define the notion of a $\cl$-phantom
extension.
\begin{citeddef}[{\cite[Definition 2.2]{Die10}}]\label{def:clphantom}
  Fix notation as in Notation \ref{notation:closure}.
  Suppose $\cl$ satisfies axioms
  $(\ref{axioms:dietz1})$--$(\ref{axioms:dietz5})$ above.
  Let $M$ be a finitely generated $R$-module and let $\alpha\colon R \to M$ be
  an injective map of $R$-modules.
  Consider the short exact sequence
  \[
    0 \longrightarrow R \overset{\alpha}{\longrightarrow} M \longrightarrow Q
    \longrightarrow 0,
  \]
  and let $\epsilon \in \Ext^1_R(Q,R)$ be the element corresponding to this
  short exact sequence via the Yoneda correspondence.
  \par Fix a projective resolution $P_\bullet$ of $Q$ consisting of finitely
  generated projective $R$-modules $P_i$.
  We say that $\epsilon$ is \textsl{$\cl$-phantom} if
  \[
    \epsilon \in \bigl(\im\bigl(\Hom_R(P_0,R) \longrightarrow
    \Hom_R(P_1,R)\bigr)\bigr)^\cl_{\Hom_R(P_1,R)}.
  \]
  This definition does not depend on the choice of
  $P_\bullet$ by \cite[Discussion 2.3]{Die10}.
  \par We say that $\alpha$ is a \textsl{$\cl$-phantom extension} if $\epsilon$
  is $\cl$-phantom.
\end{citeddef}
We now state R.G.'s algebra axiom.
\begin{citedaxiom}[{\citeleft\citen{RG18}\citemid Axiom 3.1\citeright}]
  Fix notation as in Notation \ref{notation:closure}.
  If $\cl$ satisfies axioms $(\ref{axioms:dietz1})$--$(\ref{axioms:dietz5})$
  above, we consider the following axiom:
  \begin{enumerate}[start=8,label=$(\arabic*)$,ref=\arabic*]
    \item\label{axiom:rgalg}
      (Algebra axiom) Let $\alpha \colon R \to M$ be a map of $R$-modules. 
      If $\alpha$ is an $\cl$-phantom extension, then the map $\alpha'\colon R
      \to \Sym^2_R(M)$ where $1 \mapsto \alpha(1) \otimes \alpha(1)$
      is a $\cl$-phantom extension.
  \end{enumerate}
\end{citedaxiom}
\begin{table}[t]
  \begin{ThreePartTable}
    \begin{TableNotes}
      \item[$\star$]\label{tn:char0weakalg}
        As far as we are aware, it is open whether R.G.'s algebra axiom holds
        for these closure operations.
        However, a suitable replacement that works for our proofs does hold. See
        Remark \ref{rem:weakalgaxiom}.
      \item[$F$-finite]\label{tn:ffinite} Jiang's results hold when $R/\fm$
        is $F$-finite.
    \end{TableNotes}
    {\scriptsize
    \begin{longtable}[c]{ccccc}
      \toprule
      characteristic & closure ($\cl$)
      & Dietz closure & R.G.'s algebra axiom & Brian\c{c}on--Skoda\\
      \cmidrule(lr){1-2} \cmidrule(lr){3-5}
      \multirow{2}{*}{char.\ $p > 0$}
      & tight ($*$)
      & \cite[Ex.\ 5.4]{Die10} & \cite[Prop.\ 3.6]{RG18}
      & \cite[Thm.\ 8.1]{HH94}\\
      & plus ($+$)
      & \cite[Ex.\ 5.1]{Die10} & \cite[Prop.\ 3.11]{RG18}
      & \cite[Thm.\ 7.1]{HH95}\\
      \midrule
      \multirow{3}{*}{equal char.\ $0$} & small equational tight
      (${*\mathsf{eq}}$)
      & \cite[Ex.\ 5.7]{Die10}
      & \textbf{open}\tnotex{tn:char0weakalg}
      & \cite[Thm.\ 4.1.5]{HHchar0}\\
      & big equational tight (${*\mathsf{EQ}}$)
      & \cite[Ex.\ 5.7]{Die10}
      & \textbf{open}\tnotex{tn:char0weakalg}
      & \cite[Thm.\ 4.1.5]{HHchar0}\\
      & $\mathfrak{B}$- ($+$) & \cite[Thm.\ 4.2]{Die10} & \cite[Prop.\
      3.11]{RG18} & \cite[Thm.\ 7.14.3]{AS07}\\
      \midrule
      \multirow{3}{*}{mixed char.}
      & full extended plus ($\epf$) & \textbf{open}
      & \textbf{open}\tnotex{tn:char0weakalg} & \cite[Thm.\ 4.2]{Hei01}\\
      & weak $\epf$ ($\wepf$) &
      \multicolumn{2}{c}{\cite[Thm.\ 4.8]{Jia21}\tnotex{tn:ffinite}}
      & \cite[Thm.\ 4.2]{Hei01}\\
      & full rank 1 ($\ronef$) & \multicolumn{2}{c}{\cite[Thm.\ 4.8 and Rem.\
      4.7]{Jia21}\tnotex{tn:ffinite}} & \cite[Thm.\ 4.2]{Hei01}\\
      \bottomrule
      \insertTableNotes\\
      \caption{Some closure operations on Noetherian complete local domains.}
      \label{tab:dietzclosures}
    \end{longtable}}
  \end{ThreePartTable}
\end{table}
\par We also introduce an axiom asserting that a
closure-theoretic version of the Brian\c{c}on--Skoda theorem holds.
\begin{axiom}
  Fix notation as in Notation \ref{notation:closure}.
  If $\cl$ satisfies axiom $(\ref{axioms:dietz1})$ above, we consider the
  following axiom:
  \begin{enumerate}[start=9,label=$(\arabic*)$,ref=\arabic*]
    \item\label{axiom:brianconskoda}
      (Brian\c{c}on--Skoda-type theorem)
      Let $I \subseteq R$ be an ideal generated by at most $h$ elements.
      Then, $\overline{I^{h+k}} \subseteq (I^{k+1})^\cl$ for every integer $k
      \ge 0$.
  \end{enumerate}
\end{axiom}
\subsection{Algebra closures}
The key result we need about Dietz closures satisfying R.G.'s algebra axiom
$(\ref{axiom:rgalg})$ is
that they are related to algebra closures, which are defined as follows.
\begin{citeddef}[{\cite[Definition 2.3 and Remark 2.4]{RG16}}]
  Let $R$ be a ring, and let $S$ be an $R$-algebra.
  Let $Q \subseteq M$ be an inclusion of finitely generated $R$-modules.
  We then set
  \[
    Q_M^{\cl_S} \coloneqq 
    \Set[\big]{u \in M \given 1 \otimes u \in \im(S \otimes_R Q
    \to S \otimes_R M)}.
  \]
  We call the operation $\cl_S$ an \textsl{algebra closure}.
\end{citeddef}
\begin{examples}
  If $B$ is a big Cohen--Macaulay algebra over a Noetherian local domain, then
  the algebra closure $\cl_B$ is a Dietz closure \cite[Theorem 4.2]{Die10} and
  satisfies R.G.'s algebra axiom $(\ref{axiom:rgalg})$
  \cite[Proposition 3.11]{RG18}.
  We list two examples of such algebra closures appearing in
  Table \ref{tab:dietzclosures} for which a
  Brian\c{c}on--Skoda-type theorem $(\ref{axiom:brianconskoda})$ holds.
  \begin{enumerate}[label=$(\roman*)$,ref=\roman*]
    \item Let $R$ be an excellent biequidimensional local domain of equal
      characteristic $p > 0$, or a Noetherian local domain of equal
      characteristic $p > 0$ that is a homomorphic image of a Gorenstein local
      ring.
      Then, $R^+$ is a big Cohen--Macaulay algebra over $R$ by
      \citeleft\citen{HH92}\citemid Main Theorem 5.15\citepunct
      \citen{HL07}\citemid Corollary 2.3$(b)$\citeright.
      The associated algebra closure $\cl_{R^+}$ is called \textsl{plus
      closure} \cite[Definition 2.13]{Smi94}, and
      is denoted by $+$.
    \item Let $R$ be a Noetherian local ring of equal characteristic zero.
      Using ultraproducts, Aschenbrenner and Schoutens construct a big
      Cohen--Macaulay algebra $\mathfrak{B}(R)$ over $R$ \cite[(7.7)]{AS07}.
      The associated algebra closure $\cl_{\mathfrak{B}(R)}$ is called
      \textsl{$\mathfrak{B}$-closure} \cite[(7.13)]{AS07}, and is also denoted
      by $+$.
  \end{enumerate}
\end{examples}
We now state the following result due to R.G., which says that Dietz
closures satisfying R.G.'s algebra axiom $(\ref{axiom:rgalg})$
are contained in an algebra
closure defined by a big Cohen--Macaulay algebra.
\begin{citedprop}[{\cite[Proposition 4.1]{RG18}}]\label{prop:rg41}
  Let $R$ be a Noetherian local domain and let $\cl$ be a Dietz closure on $R$
  that satisfies R.G.'s algebra axiom $(\ref{axiom:rgalg})$.
  Then, there exists a big Cohen--Macaulay algebra $B$ over $R$ such that
  $Q^\cl_M \subseteq Q^{\cl_B}_M$
  for every inclusion $Q \subseteq M$ of finitely generated $R$-modules.
\end{citedprop}

\section{Proofs of Theorems
\texorpdfstring{\ref{thm:mainelshhms}}{\ref*{thm:mainelshhms}},
\texorpdfstring{\ref{thm:elshhmsmb}}{\ref*{thm:elshhmsmb}}, and
\texorpdfstring{\ref{thm:elshhmsmc}}{\ref*{thm:elshhmsmc}}}\label{sect:proofs}
Theorem \ref{thm:mainelshhms} follows from
Theorem \ref{thm:elshhmsmb} by setting $s_i = 0$ for all $i$.
It therefore suffices to show Theorems \ref{thm:elshhmsmb} and
\ref{thm:elshhmsmc}.\medskip
\par We start with a version of \cite[Lemma 4.1]{Mur} that allows us to reduce
to the local case where the residue field is perfect, and hence $F$-finite if it
is of characteristic $p > 0$.
\begin{lemma}\label{lem:41ffin}
  It suffices to show Theorems \ref{thm:elshhmsmb} and \ref{thm:elshhmsmc}
  under the additional assumptions that $R$ is a complete local ring
  with a perfect residue field, and that
  the localizations of $R$ at the associated primes of $R/I$ have infinite
  residue fields.
\end{lemma}
\begin{proof}
  By \cite[Lemma 4.1, Steps 1 and 2]{Mur}, we may assume that $R$ is a local
  ring $(R,\fm)$ such that the localizations of $R$ at the associated primes of
  $R/I$ have infinite residue fields.
  We now want to enlarge the residue field of $R$ to be perfect.
  By \cite[Chapitre IX, Appendice, n\textsuperscript{o} 2, Corollaire to
  Th\'eor\`eme 1]{BouAC89}, there exists a gonflement $(R,\fm) \to
  (R',\fm')$ where $R'/\fm'$ is a perfect closure of $R/\fm$.
  This map is a flat local map such that $R'$ is regular
  \cite[Chapitre IX, Appendice, n\textsuperscript{o} 2, Proposition 2]{BouAC89}.
  We now consider the composition $R \to R' \to \widehat{R'}$,
  where the second map is the $\fm'$-adic completion map.
  This map is faithfully flat, and
  \cite[Lemma 4.1, Step 3]{Mur}
  shows that we may replace $R$ by $\widehat{R'}$.
\end{proof}
Next, we prove the following closure-theoretic analogue of \cite[Theorem
D]{Mur}.
\begin{proposition}\label{prop:elshhmsmd}
  Let $(R,\fm)$ be a complete local domain.
  Consider an ideal $I \subseteq R$.
  Let $u \in I^{(M)}$ for an integer $M > 0$, and fix non-negative integers
  $s_1,s_2,\ldots,s_n$.
  \begin{enumerate}[label=$(\roman*)$,ref=\roman*]
    \item\label{prop:elshhmsmdrprime}
      Denote by $R'$ the integral closure of
      \[
        R[u^{1/M}] \subseteq \Frac(R)[u^{1/M}]
      \]
      in $\Frac(R)[u^{1/M}]$,
      where $u^{1/M}$ is a choice of $M$-th root of $u$ in an algebraic closure
      of $\Frac(R)$.
      Then, $R'$ is a complete normal local domain $(R',\fm')$.
      Moreover, if $R/\fm$ is $F$-finite and of characteristic $p > 0$, then
      $R'/\fm'$ is $F$-finite and of characteristic $p > 0$.
  \end{enumerate}
  Now assume that the localizations of $R$ at
  the associated primes of $R/I$ have infinite residue fields.
  Let $h$ be the largest analytic spread of $IR_\fp$, where $\fp$ ranges
  over all associated primes $\fp$ of $R/I$.
  \begin{enumerate}[resume,label=$(\roman*)$,ref=\roman*]
    \item\label{prop:elshhmsmdepf}
      Fix a closure operation $\cl$ on $R$ for which the
      Brian\c{c}on--Skoda-type theorem $(\ref{axiom:brianconskoda})$ holds.
      Let $\fp_\ell \in \Ass_R(R/I)$.
      Then, there exists $c_\ell \in R - \fp_\ell$ such that for every $i$, we
      have
      \[
        c_\ell u^{\frac{s_i+h}{M}} \in (I^{s_i+1}R')^\cl.
      \]
    \item\label{prop:elshhmsmdbcm}
      Suppose that if $R$ is of mixed characteristic, then $R/\fm$ is
      $F$-finite and of characteristic $p > 0$.
      Then, there exists a big Cohen--Macaulay algebra $B$ over $R'$
      such that the following holds:
      For every $\fp_\ell \in \Ass_R(R/I)$, there exists $c_\ell \in R -
      \fp_\ell$ such that for every $i$, we have
      \[
        c_\ell u^{\frac{s_i+h}{M}} \in I^{s_i+1}B.
      \]
      Moreover, if $R$ is regular, then for every $i$, we have
      \[
        u^{\frac{s_i+h}{M}} \in I^{(s_i+1)}B.
      \]
  \end{enumerate}
\end{proposition}
\begin{proof}
  We first prove $(\ref{prop:elshhmsmdrprime})$.
  The ring $R'$ is module-finite over $R$ and is complete local by
  \cite[Chapitre 0, Th\'eor\`eme 23.1.5 and Corollaire 23.1.6]{EGAIV1}.
  If $R/\fm$ is $F$-finite and of characteristic $p > 0$, then $R'/\fm'$ is
  $F$-finite and of characteristic $p > 0$
  because the field extension $R/\fm \hookrightarrow R'/\fm'$ is a finite field
  extension.\smallskip
  \par Next, we prove $(\ref{prop:elshhmsmdepf})$.
  Fix $\fp_\ell \in \Ass_R(R/I)$.
  Since the residue field of $R_{\fp_\ell}$ is infinite, there exists an ideal
  $J \subseteq I$ with at most $h$ generators such that $JR_{\fp_\ell}$ is a
  reduction of $IR_{\fp_\ell}$ by \cite[Proposition 8.3.7]{SH06}.
  By \cite[Lemma 8.1.3(1)]{SH06}, we also know that $JR'_{\fp_\ell}$ is a
  reduction of $IR'_{\fp_\ell}$, and
  by \cite[Proposition 8.1.5]{SH06}, we know that
  $J^{s_i+h}R_{\fp_\ell}'$ is a reduction of $I^{s_i+h}R'_{\fp_\ell}$ for all
  $i$.
  \par Since $u \in I^{(M)}$, there exists an element $x_\ell \in R - \fp_\ell$
  such that $x_\ell u \in I^M$.
  We then have
  \begin{align*}
    \bigl(x_\ell^{s_i+h}u^{\frac{s_i+h}{M}}\bigr)^{M} =
    x_\ell^{M(s_i+h)}u^{s_i+h}
    &= (x_\ell^M u)^{s_i+h}
    \in (I^{s_i+h}R')^{M},
  \intertext{and hence}
    x_\ell^{s_i+h}u^{\frac{s_i+h}{M}} &\in
    \overline{I^{s_i+h}R'}.
  \intertext{Since $\overline{J^{s_i+h}R_{\fp_\ell}'} =
  \overline{I^{s_i+h}R'_{\fp_\ell}}$, there exists an element $y_{\ell,i}
  \in R-
  \fp_\ell$ such that}
  y_{\ell,i} x_\ell^{s_i+h}u^{\frac{s_i+h}{M}} &\in
    \overline{J^{s_i+h}R'}.
  \end{align*}
  By the Brian\c{c}on--Skoda-type theorem $(\ref{axiom:brianconskoda})$,
  we have
  \[
    y_{\ell,i} x_\ell^{s_i+h}u^{\frac{s_i+h}{M}} \in
    (J^{s_i+1}R')^\cl \subseteq (I^{s_i+1}R')^\cl,
  \]
  where the inclusion on the right holds by the order-preservation axiom
  $(\ref{axioms:dietz3})$.
  Setting $y_\ell = y_{\ell,1}y_{\ell,2}\cdots y_{\ell,n}$, we therefore have
  \[
    y_\ell x_\ell^{s+h}u^{\frac{s_i+h}{M}} \in (I^{s_i+1}R')^\cl
  \]
  for all $i$.
  Setting $c_\ell = y_\ell x_\ell^{s+h}$, we obtain
  $(\ref{prop:elshhmsmdepf})$.\smallskip
  \par Finally, we show $(\ref{prop:elshhmsmdbcm})$.
  Fix a Dietz closure $\cl$ on $R$ satisfying R.G.'s algebra axiom and the
  Brian\c{c}on--Skoda-type theorem $(\ref{axiom:brianconskoda})$.
  Note that such a closure operation exists in all characteristics by Table
  \ref{tab:dietzclosures}.
  By $(\ref{prop:elshhmsmdepf})$, there exist $c_\ell \in R - \fp_\ell$ such
  that
  \begin{align*}
    c_\ell u^{\frac{s_i+h}{M}} &\in (I^{s_i+1}R')^\cl
    \intertext{for all $i$.
    By R.G.'s result stating that Dietz closures satisfying R.G.'s algebra
    axiom are contained in a big Cohen--Macaulay algebra closure
    (Proposition \ref{prop:rg41}),
    there exists a big Cohen--Macaulay algebra $B$
    over $R'$ such that}
    c_\ell u^{\frac{s_i+h}{M}} &\in I^{s_i+1}B.
  \end{align*}
  This proves the first statement in $(\ref{prop:elshhmsmdbcm})$.
  \par For the second statement in $(\ref{prop:elshhmsmdbcm})$,
  for each $\ell$ such that $\fp_\ell$ is maximal in $\Ass_R(R/I)$ with respect
  to inclusion, choose $d_\ell \in \fp_\ell - \bigl(\bigcup_{j \ne \ell}
  \fp_j \bigr)$, which is possible by prime avoidance.
  As in the proof of Lemma \ref{lem:symbolicdefs}, we have
  \[
    c = \sum_{\ell} \biggl(c_\ell \prod_{j \ne \ell}d_j\biggr) \notin
    \bigcup_\ell \fp_\ell.
  \]
  We also have
  \[
    cu^{\frac{s_i+h}{M}} \in I^{s_i+1}B \subseteq I^{(s_i+1)}B.
  \]
  We now note that $B$ is a big Cohen--Macaulay algebra over $R$, since
  every system of parameters in $R$ maps to a system of parameters in $R'$.
  Since $R$ is regular, $R \to B$ is therefore faithfully flat by \cite[Lemma
  2.1$(d)$]{HH95}.
  Thus, since $c$ is a nonzerodivisor on $R/I^{(s_i+1)}$ by \cite[Theorem
  6.1$(ii)$]{Mat89}, it is also a
  nonzerodivisor on $B/I^{(s_i+1)}B$.
  We therefore see that for every $i$, we have
  \[
    u^{\frac{s_i+h}{M}} \in I^{(s_i+1)}B.\qedhere
  \]
\end{proof}
\begin{remark}\label{rem:weakalgaxiom}
  The proof of Proposition \ref{prop:elshhmsmd}$(\ref{prop:elshhmsmdbcm})$ in
  fact applies to any closure operation satisfying a Brian\c{c}on--Skoda-type
  theorem $(\ref{axiom:brianconskoda})$ and for which the
  following property holds for all inclusions $Q\subseteq M$ of finitely
  generated $R$-modules:
  \begin{enumerate}[label=$(\star)$,ref=\star]
    \item\label{axiom:weakalg}
      If $u \in Q^\cl_M$, then there exists a big Cohen--Macaulay
      algebra $B$ over $R$ such that $1 \otimes u \in \im(B \otimes_R Q \to
      B \otimes_R M)$.
  \end{enumerate}
  The property $(\ref{axiom:weakalg})$
  holds for all Dietz closures satisfying R.G.'s algebra axiom
  $(\ref{axiom:rgalg})$ by 
  R.G.'s result stating that Dietz closures satisfying R.G.'s algebra
  axiom are contained in a big Cohen--Macaulay algebra closure
  (Proposition \ref{prop:rg41}),
  but holds for the other closure
  operations listed in Table \ref{tab:dietzclosures} as well.
  \begin{enumerate}[label=$(\roman*)$]
    \item The property $(\ref{axiom:weakalg})$ holds for big (and hence also
      small) equational tight closure on all
      Noetherian rings of equal characteristic zero \cite[Theorem 11.4]{Hoc94}.
      See \cite[Definition 3.4.3$(b)$ and (4.6.1)]{HHchar0} for definitions of
      these closure operations.
    \item The property $(\ref{axiom:weakalg})$ holds for $\epf$ closure for
      complete local domains of mixed characteristic with $F$-finite residue
      field.
      This follows from the fact that $\wepf$ closure is a Dietz closure
      satsifying R.G.'s algebra axiom $(\ref{axiom:rgalg})$
      \cite[Theorem 4.8]{Jia21} since $Q_M^\epf \subseteq Q_M^\wepf$.
    \item A stronger version of $(\ref{axiom:weakalg})$ holds for tight
      closure on all analytically irreducible excellent local domains of
      equal characteristic $p > 0$ \cite[Theorem 11.1]{Hoc94}.
  \end{enumerate}
  We can therefore use these closure operations and
  property $(\ref{axiom:weakalg})$ instead of Proposition
  \ref{prop:rg41} to prove Theorems \ref{thm:elshhmsmb} and \ref{thm:elshhmsmc}
  below.
\end{remark}
Finally, we can prove Theorems \ref{thm:elshhmsmb} and \ref{thm:elshhmsmc}.
\begin{proof}[Proof of Theorems \ref{thm:elshhmsmb} and \ref{thm:elshhmsmc}]
  By Lemma \ref{lem:41ffin}, we may assume that $R$ is a complete regular local
  ring $(R,\fm)$ with a perfect residue field,
  and that the localizations of $R$ at the associated primes of $R/I$ have
  infinite residue fields.\smallskip
  \par We start with Theorem \ref{thm:elshhmsmb}.
  Setting $M = s+nh$ in 
  Proposition \ref{prop:elshhmsmd}$(\ref{prop:elshhmsmdbcm})$, for $u \in
  I^{(s+nh)}$, we have
  \[
    u^{\frac{s_i+h}{s+nh}} \in I^{(s_i+1)}B
  \]
  for some big Cohen--Macaulay algebra $B$ over $R$.
  Multiplying together these inclusions for every $i$, we have
  \[
    u \in \prod_{i=1}^n I^{(s_i+1)}B.
  \]
  Since $R$ is regular, $R \to B$ is faithfully flat by \cite[Lemma
  2.1$(d)$]{HH95}, and hence we have
  \[
    u \in \prod_{i=1}^n I^{(s_i+1)}.\smallskip
  \]
  \par We now prove Theorem \ref{thm:elshhmsmc}.
  Setting $M = s+nh+1$ in 
  Proposition \ref{prop:elshhmsmd}$(\ref{prop:elshhmsmdbcm})$, for $u \in
  I^{(s+nh+1)}$, we have
  \[
    u^{\frac{s_i+h}{s+nh+1}} \in I^{(s_i+1)}B
  \]
  for some big Cohen--Macaulay algebra $B$ over $R$.
  Multiplying together these inclusions for every $i$, we have
  \[
    u^{\frac{s+nh}{s+nh+1}} \in \prod_{i=1}^n I^{(s_i+1)}B.
  \]
  \par Let $l$ be the largest integer for which the composition
  \[
    R \longrightarrow B \xrightarrow{u^{\frac{l}{s+nh+1}} \cdot -} B
  \]
  splits as a map of $R$-modules.
  We claim such an $l$ exists.
  When $l = 0$, the map $R \to B$ is faithfully
  flat, hence pure \cite[Theorem 7.5$(i)$]{Mat89}.
  Thus, the map $R \to B$ splits by an argument of Auslander \cite[p.\
  59]{HH90}.
  \par We now have
  \[
    u^{\frac{l}{s+nh+1}}u =
    u^{\frac{l+1}{s+nh+1}}u^{\frac{s+nh}{s+nh+1}} \in u^{\frac{l+1}{s+nh+1}}
    \cdot \prod_{i=1}^n I^{(s_i+1)}B.
  \]
  Applying $\varphi \in \Hom_R(B,R)$, we have
  \[
    \varphi\Bigl(u^{\frac{l}{s+nh+1}}\Bigr) \cdot u =
    \varphi\Bigl(u^{\frac{l+1}{s+nh+1}}u^{\frac{s+nh}{s+nh+1}}\Bigr)
    \in
    \varphi\Bigl(u^{\frac{l+1}{s+nh+1}}\Bigr) \cdot \prod_{i=1}^n
    I^{(s_i+1)}.
  \]
  Summing over all $\varphi$, we obtain
  \[
    u \in \fm \cdot \prod_{i=1}^n I^{(s_i+1)}
  \]
  as claimed, since $\varphi(u^{\frac{l+1}{s+nh+1}}) \in \fm$ for all
  $\varphi \in \Hom_R(B,R)$ by the assumption on $l$.
\end{proof}

\end{document}